\documentclass{article}

\usepackage{amsmath, amsthm, amssymb, amstext, amsfonts}
\usepackage{enumerate}
\usepackage{graphicx}
\usepackage[applemac]{inputenc}

\theoremstyle{plain}

\newtheorem{thm}{Theorem}[section]
\newtheorem{lem}[thm]{Lemma}
\newtheorem{cor}[thm]{Corollary}
\newtheorem{prop}[thm]{Proposition}

\newtheorem{clm}[thm]{Claim}

\theoremstyle{definition}

\newcommand{\sm}{\ensuremath{\setminus}}

\newcommand{\set}[1]{\ensuremath{\{#1\}}}

\newcommand{\diam}{\textnormal{diam}}
\newcommand{\rad}{\textnormal{rad}}
\newcommand{\asdim}{\textnormal{asdim}}
\newcommand{\isom}{\ensuremath{\cong}}
\newcommand{\inv}{\ensuremath{^{-1}}}

\newcommand{\rand}{\partial}

\newcommand{\sub}{\subseteq}

\newcommand{\seq}[3]{(#1_#2)_{#2\in #3}}

\newcommand{\nat}{{\mathbb N}}
\newcommand{\real}{{\mathbb R}}

\newcommand{\BF}{\ensuremath{\mathcal B}}

\newcommand{\PF}{\ensuremath{\mathcal P}}

\newcommand{\TF}{\ensuremath{\mathcal T}}
\newcommand{\UF}{\ensuremath{\mathcal U}}
\newcommand{\VF}{\ensuremath{\mathcal V}}

\newcommand{\YF}{\ensuremath{\mathcal Y}}

\begin{document}

\title{On the tree-likeness of hyperbolic spaces}
\author{Matthias Hamann\thanks{Partially supported by FWF grant P-19115-N18.}\bigskip \\Fachbereich Mathematik\\Universit\"at Hamburg}
\date{\today}
\maketitle
\begin{abstract}
In proper hyperbolic geodetic spaces we construct rooted $\real$-trees with the following properties.
On the one hand, every ray starting at the root is quasi-geodetic; so these $\real$-trees represent the space itself well.
At the same time, the trees boundary reflects the boundary of the space in that the number of disjoint rays to a boundary point is bounded in terms of the (Assouad) dimension of the hyperbolic boundary.
\end{abstract}

\section{Introduction}

Since Gromov's article on hyperbolic groups \cite{gromov} appeared, there have been various attempts to describe a given hyperbolic space by comparing it with the simplest form of a hyperbolic space, an $\real$-tree.

On the one hand, there are results that construct for a given hyperbolic space an $\real$-tree whose local structure resembles the local structure of the hyperbolic space.
The best known among these are results attributed to Gromov (see \cite[Chapitre~8]{CoornDelPapa} and \cite[\S2.2]{GhHaSur}) that construct for a finite subset of the completion of a $\delta$-hyperbolic space an $\real$-tree in the space whose completion contains the given set and such that all of its geodesics between elements of the finite set are quasi-geodesics in the hyperbolic space for constants that depend only on the size of the set and on~$\delta$.
There is also a result by Benjamini and Schramm~\cite[Theorem 1.5]{BenjaminiSchramm-Cheeger} for locally finite hyperbolic graphs which states that there exists a subtree with exponential growth such that the embedding is a bilipschitz map.

On the other hand there are constructions of $\real$-trees that try to capture the boundary of a given hyperbolic space.
When the space is a locally finite graph, then several ideas for such constructions can already be found in Gromov's article~\cite[Sections 7.6, 8.5.B, and 8.5.C]{gromov}.
They have been elaborated on in~\cite[Chapter~5]{SymbolicDynamics}.
These trees capture the boundary of the hyperbolic graph in that there is a continuous map from their own boundary onto that of the graph.
However, these trees are not necessarily subtrees of the hyperbolic graph.
If the hyperbolic graph has bounded degree, then some of these maps are finite-to-one.
In~\cite{HyperbolicSpanningTrees} the author showed that inside every hyperbolic graphs of bounded degree there exists a spanning whose boundary maps continuously and finite-to-one onto the boundary of the graph.

\medskip

In this article, we combine these two approaches.
Inside every proper hyperbolic geodetic space $X$ whose boundary has finite Assouad dimension we construct a rooted $\real$-tree $T$ such that
\begin{enumerate}[$\bullet$]
\item the rays from the root are all quasi-geodetic rays (for the same global constants) and
\item a continuous finite-to-one map from the boundary of the tree onto the one of the hyperbolic space exists where the bound only depends on the Assouad dimension of the hyperbolic boundary.
\end{enumerate}

The assumption that the hyperbolic boundary has finite Assouad dimension is not a strong assumption.
For example, if the space is a locally finite hyperbolic graph of bounded degree, e.g.\ if the graph is the Cayley graph of a finitely generated hyperbolic group with a finite set of generators.
If the hyperbolic space is visual, that is roughly speaking that every points has bounded distance to a geodetic double ray (see Section~\ref{sec_DefVisual} for more details), then every point of the space has distance at most some constant $\kappa$ from the $\real$-tree (Theorem~\ref{thm_mainPart1}).
If we consider an arbitrary proper hyperbolic geodetic space, then a $\kappa$ exists with the property that every geodesic outside the described set has finite length (Theorem~\ref{thm_mainPart2}).

\medskip

Different approaches exhibiting the tree-likeness of hyperbolic spaces include quasi-isometric embeddings of visual hyperbolic spaces into the product of $\real$-trees, see Buyalo et al.\ \cite{BDS-Embedding}, or the sub-cones at infinity, see~\cite[Proposition~2.1.11]{GhHaSur} and \cite[Lemme~5.6]{Paulin}.

\medskip

In the final section we give a proof that for any continuous function from the boundary of an $\real$-tree to the boundary of a hyperbolic space the number of preimages of a boundary point of the space is bounded from below by a function that depends only on the topological dimension of the hyperbolic space.

\section{Hyperbolic spaces}\label{hypGraphs}\label{sec_hypGraphs}

In this section we define properties for metric spaces, in particular, for hyperbolic spaces and cite some of their properties. For a more detailed introduction to hyperbolicity, we refer to \cite{ABCFLMSS,CoornDelPapa,GhHaSur,gromov,Ohshika} as well as \cite[Chapter~III.H]{BridsonHaefliger} and \cite[Chapter~22]{woessBook}.

\medskip

Let $X$ be a metric space.
A {\em geodesic} between two points $x,y\in X$ is the image of an isometric map $\varphi:[0,d(x,y)]\to X$ with $\varphi(0)=x$ and $\varphi(d(x,y))=y$.
With $[x,y]$ we denote any geodesic between $x$ and $y$.
If we want to specify the particular metric space $X$, then we write $[x,y]_X$.
The space $X$ is {\em geodetic} if for any two points $x,y\in X$ there exists a geodesic in~$X$ between them.
It is {\em proper} if for every $r\in\nat$ and $x\in X$ the closed ball $\bar{B}_r(x)$ is compact.
If there is a $\delta\geq 0$ such that for any three points $x,y,z$ and any geodesics $[x,y],[y,z],[z,x]$ between each two of the points there is $[x,z]\sub \bar{B}_\delta([x,y]\cup[y,z])$ then we call the space {\em ($\delta$-)hyperbolic} and $\delta$ is the {\em hyperbolicity constant}.

\medskip

Homeomorphic images of~$[0,1]$ are called {\em paths}.
A {\em ray} is a homeomorphic image $R$ of~$[0,\infty)$ such that for every ball of finite diameter $R$ lies eventually outside that ball.
{\em Double rays} are homeomorphic images of~$\real$ such that the restrictions to~$\real_{\geq0}$ and to~$\real_{\leq0}$ are both rays.
A (double) ray is {\em geodetic} if it is an isometric image of~$[0,\infty)$ (of~$\real$).
A ray is {\em eventually geodetic} if it has a geodetic subray.

Since we are looking at the hyperbolic boundary from distinct viewpoints, we state here three different definitions of the hyperbolic boundary all of which are equivalent.
Two geodetic rays $\pi_1,\pi_2$ are {\em equivalent} if for any sequence $(x_n)_{n\in\nat}$ of points on~$\pi_1$ we have $\liminf_{n\to\infty}d(x_n,\pi_2)\leq M$ for an $M<\infty$.
In hyperbolic geodetic spaces, this is an equivalence relation, compare with \cite[Section~2.4.2]{BS-Elements}.
The {\em hyperbolic boundary} $\rand X$ is the set of all equivalence classes of this relation.
With $\widehat{X}$ we denote $X\cup\rand X$.

The {\em Gromov-product} (with respect to $o\in X$) of two elements $x,y\in X$ is
$$(x,y)_o:=\frac{1}{2}(d(x,o)+d(y,o)-d(x,y)).$$
If it is obvious by the context which point we use as the base-point for the product, we simply write $(x,y)$.

Now we give the second topological definition of the hyperbolic boundary.
A sequence $(x_i)_{i\geq 0}$ {\em converges to a point} $x$ if $\lim_{i\to\infty}(x_i,x)=0$.
A sequence $(x_i)_{i\geq 0}$ {\em converges to infinity} if $\lim_{i,j\to\infty}(x_i,x_j)\to\infty$.
Two sequences $(x_i)_{i\geq 0}$, $(y_j)_{j\geq 0}$ that converge to infinity are {\em equivalent} if $\lim_{i,j\to\infty}(x_i,y_j)=\infty$.
In hyperbolic geodetic spaces this equivalence is an equivalence relation.
The {\em hyperbolic boundary} is the set of equivalence classes of this equivalence relation.
A sequence $(x_i)_{i\geq 0}$ {\em tends} to a boundary point $\eta$ if it is in the equivalence class $\eta$ (notation: $(x_i)_{i\geq 0}\to\eta$).
In~\cite{GhHaSur} the equivalence of this definition with the first one given is shown.

A third way to define the hyperbolic boundary is via the completion defined by a metric $d_\varepsilon$.
Let $\varepsilon>0$ with $\varepsilon':=\exp(\varepsilon\delta)-1\leq\sqrt{2}-1$.
Let
\[d_\varepsilon(x,y):=\inf\set{\sum_{i=1}^n\exp(-\varepsilon(x_{i-1},x_i))\mid x_i\in X, x_0=x,x_n=y}.\]
Then $d_\varepsilon$ is a metric on~$X$.
The {\em hyperbolic boundary} is the completion of~$X$ with respect to this metric without~$X$.
For a proof of the equivalence of this definition with the previous ones see~\cite[Proposition~7.3.10]{GhHaSur}.

An important theorem about the hyperbolic boundary is the following.
For references see~\cite[Section~6]{BonkSchramm_Embeddings}, \cite[Proposition~2.3.2]{CoornDelPapa}, and \cite[Corollary~2.65]{Ohshika}.

\begin{thm}\label{BoundComp}\label{MetricTopologyOfBoundary}
If $X$ is a proper geodetic hyperbolic space, then the hyperbolic boundary is compact for all metrics $d_\varepsilon$ with $\varepsilon>0$ and $\exp(\varepsilon\delta)\leq\sqrt{2}$.

Furthermore, for all $\eta,\mu\in\rand X$ and with $\varepsilon'=\exp(\varepsilon\delta)-1$ there is

\hfill$\,\,\varepsilon'\exp(-\varepsilon\,(\eta,\mu))\le d_\varepsilon(\eta,\mu)\leq\exp(-\varepsilon\,(\eta,\mu)).$\qed
\end{thm}

Geodetic (double) rays play an important role in the context of hyperbolic geodetic spaces, as we already saw in the first definition of the hyperbolic boundary.
The following proposition shows that there are plenty of them.

\begin{prop}\label{prop_geodExists}{\em \cite[Proposition~2.60]{Ohshika}, \cite[Proposition~2.2.1]{CoornDelPapa}}
Let $X$ be a proper hyperbolic geodetic space.
For every $x\in X$ and every hyperbolic boundary point there is a geodetic ray from $x$ to that boundary point, and for every two distinct hyperbolic boundary points there is a geodetic double ray between these two boundary points.\qed
\end{prop}

Let $\gamma>1,c\geq 0$.
A {\em $(\gamma,c)$-quasi-isometry} from $X$ to another metric space $Y$ is a map $f:X\to Y$ with
$$\gamma^{-1}d_X(x,y)-c\leq d_Y(f(x),f(y))\leq\gamma d_X(x,y)+c$$
for all $x,y\in X$ and with $\sup\{d_Y(y,f(X))\mid y\in Y\}\leq b$.
Then $X$ is {\em quasi-isometric} to~$Y$.
A (double) ray $R$ is {\em $(\gamma,c)$-quasi-geodetic} if it is the image of a $(\gamma,c)$-quasi-isometry from $\real_{\geq 0}$ ($\real$, resp.)\ to~$R$.
Hence a (double) ray is geodetic, if it is a $(1,0)$-quasi-geodetic (double) ray.
If the constants $\gamma,c$ are not important, then we just speak of {\em quasi-geodesics}.

The next proposition shows that in every proper hyperbolic geodetic space the geodesics and quasi-geodesics lie close to each other.

\begin{prop}\label{prop_geodAndQuasigeod}{\em \cite[Theorem~2.31]{Ohshika}, \cite[Th\'eor\`eme~3.1.4]{CoornDelPapa}}
Let $X$ be a proper $\delta$-hyperbolic geodetic space.
For all $\gamma_1\geq 1,\gamma_2\geq 0$ there is a constant $\kappa=\kappa(\delta,\gamma_1,\gamma_2)$ such that for every two points $x,y\in X$ every $(\gamma_1,\gamma_2)$-quasi-geodesic between them lies in a $\kappa$-neighborhood around every geodesic between $x$ and $y$ and vice versa.

Furthermore, this extends to $(\gamma_1,\gamma_2)$-quasi-geodetic and geodetic (double) rays.\qed
\end{prop}

\begin{prop}\label{LimesOfGromovProd}{\em \cite[(22.4)]{woessBook}}
Let $X$ be a proper $\delta$-hyperbolic geodetic space.
Then for all $x,y,z\in X$ we have

\hfill $\,\,(x,y)_z\leq d(z,[x,y])\leq (x,y)_z+2\delta.$\qed
\end{prop}

We may extend to definition of the Gromov-product to~$\widehat{X}$.
For $a,b\in\widehat{X}$ let
$$(a,b):=\inf\liminf_{i,j\to\infty}(x_i,y_j)$$
where the infimum is taken over all sequences $(x_i)_{i\geq 0}\to a$ and $(y_i)_{i\geq 0}\to b$.

Combining Proposition~\ref{prop_geodAndQuasigeod} and \cite[Lemma~2.2.2]{BS-Elements} we obtain the following proposition.

\begin{prop}\label{geodIn4Delta}
Let $X$ be a proper $\delta$-hyperbolic geodetic space, let $\eta,\nu\in\rand X$, and let $o\in X$.
For all geodetic double rays $\pi$ from $\eta$ to~$\nu$ we have

\hfill$\,\,(\eta,\nu)_o\leq d(o,\pi)\leq (\eta,\nu)_o+4\delta.$\qed
\end{prop}

\begin{prop}\label{prop_QuotientLeadsToDistance}
Let $X$ be a proper geodetic hyperbolic space with a metric $d_\varepsilon$ as in Theorem~\ref{BoundComp} with $\varepsilon>0$ and $\varepsilon':=\exp(\varepsilon\delta)-1\leq\sqrt{2}-1$.
Let $o\in X$ be the base-point for the Gromov-product of~$X$.
Then, for every $q>0$, there exists a $\beta=\beta(q,\varepsilon)>0$ such that for all $\eta_1,\eta_2,\mu_1,\mu_2\in\rand X$ with $\frac{1}{q}\leq d_\varepsilon(\eta_1,\mu_1)/d_\varepsilon(\eta_2,\mu_2)\leq q$ there is $|d(o,[\eta_1,\mu_1])-d(o,[\eta_2,\mu_2])|\leq \beta$.
\end{prop}

\begin{proof}
By Theorem~\ref{BoundComp} there is
$$\varepsilon'\exp(-\varepsilon(\eta_1,\mu_1))\leq d_\varepsilon(\eta_1,\mu_1)\leq q d_\varepsilon(\eta_2,\mu_2)\leq q\exp(-\varepsilon(\eta_2,\mu_2)).$$
As a consequence we have by symmetry
$$|(\eta_1,\mu_1)-(\eta_2,\mu_2)|\leq\frac{1}{\varepsilon}\ln(\frac{q}{\varepsilon'}).$$
The claim follows immediately with Proposition~\ref{geodIn4Delta}.
\end{proof}

An {\em $\real$-tree} is a metric space $T$ such that for any two points $x,y\in T$ there exists a unique arc between them that is isometric to the interval $[0,d(x,y)]$.
An observation is that $\real$-trees are $0$-hyperbolic geodetic spaces.
The converse direction---that $0$-hyperbolic geodetic spaces are $\real$-trees---is a bit more difficult.
But proofs can be found in nearly every of the introductory books or articles on hyperbolic spaces.
For more details on $\real$-trees see for example~\cite{C-Arbres, H-TreesUltrametrics, H-TreesUltrametrics2}.

\section{The Assouad dimension}\label{Topology}\label{sec_Topology}

In this section we introduce the Assouad dimension, which is the main dimension concept in this article.
Furthermore, we compare it with related concepts.
For a more detailed introduction to the Assouad dimension we refer to~\cite{Assouad} and in particular to~\cite[Appendix~A]{Luukainen}.

Let $X$ be a metric space.
For $\alpha,\beta>0$ let $S(\alpha,\beta)$ be the maximal cardinality of a subset $V$ of $X$ such that each two distinct elements of~$V$ have distance at least $\alpha$ and at most $\beta$.
Let $n$ be the infimum of all $s\geq 0$ such that there is a $C\ge 0$ with $S(\alpha,\beta)\leq C(\frac{\beta}{\alpha})^s$ for all $0<\alpha\leq\beta$.
Then $n$ is called the {\em Assouad dimension} of the metric space $X$ (notation: $\dim_A(X)=n$).

A metric space $X$ is {\em doubling} if there exists a $\kappa\geq 1$ such that every ball of radius $r$ can be covered by at most $2^\kappa$ balls of radius at most $\frac{r}{2}$.
With $\dim_2(X)$ we denote the infimum of all these $\kappa$.
A subset $Y$ of $X$ has {\em diameter} $\sup\set{d(x,y)\mid x,y\in Y}$ (notation: $\diam(Y)$), and a set $\YF\sub\PF(X)$ has {\em diameter} $\diam(\YF)=\sup\set{\diam(Y)\mid Y\in\YF}$.
The {\em radius} of a subset $Y$ of~$X$ is $\rad(Y):=\inf\{\sup\{d(x,y)\mid x\in Y\}\mid Y\in Y\}$ and the {\em radius} of a set $\YF\sub\PF(X)$ is $\rad(\YF):=\sup\{\rad(Y)\mid Y\in\YF\}$.
For every $r\geq 0$, a family $\BF=(B_i)_{i\in I}$ of subsets of~$X$ has {\em $r$-multiplicity} at most $n$ if every subset of~$X$ with diameter at most $r$ intersects with at most $n$ members of the family.
A point $x\in X$ has {\em $r$-multiplicity} at most $n$ in~$\BF$ if $\bar{B}_r(x)$ intersects with at most $n$ members of the family $\BF$ non-trivially.

Our main assumption is that the Assouad dimension of the hyperbolic boundary of our hyperbolic space is finite.
It is easier to use the doubling property instead.
The following theorem guarantees that we treat the same spaces.

\begin{thm}\label{AssouadDoubling}{\em \cite[Theorem A.3]{Luukainen}}
Let $X$ be a metric space.
Then, $X$ is doubling if and only if it has finite Assouad dimension.\qed
\end{thm}

It is easy to adapt the proof of \cite[Lemma 2.3]{LS-Nagata} for Lemma~\ref{LS2.3}, see~\cite[Lemma~3.2]{HyperbolicSpanningTrees} for details.

\begin{lem}\label{LS2.3}
Let $X$ be a doubling metric space, let $N=2^{\dim_2(X)}$, and let $r>0$. Then $X$ has a covering $\BF$ of closed balls of radius $r$ such that $\BF$ is the disjoint union of at most $N^2$ subsets $\BF_i$ of $\BF$ each of which has $r$-multiplicity at most $1$; so $\BF$ has $r$-multiplicity at most $N^2$.

Furthermore, it is possible to choose a given subset $Y$ of $X$ with $d(x,y)>r$ for all $x,y\in Y$ so that $Y$ is a subset of the set of centers of the balls in $\BF$ and such that each two centers of these balls have distance at least $r$ and each center has $3r$-multiplicity at most $N^2$ in~$\BF$.
\qed
\end{lem}

Let us briefly compare the Assouad dimension with another dimension concept.
A metric space $X$ has {\em asymptotic dimension} $n$ (notation: $\asdim (X)=n$) if $n$ is the smallest natural number such that for every $\varrho>0$ there exists an open cover $\UF$ of~$X$ such that every $x\in X$ lies in at most $n+1$ elements of~$\UF$, such that $\sup_{U\in\UF}\diam(U)<\infty$, and such that
$$\inf_{x\in X}\sup_{U\in\UF}d(x,X\sm U)\geq \varrho.$$

In the main theorems (Theorem~\ref{thm_mainPart1} and Theorem~\ref{thm_mainPart2}) we are talking about proper hyperbolic geodetic spaces whose hyperbolic boundary has finite Assoud dimension.
Since the hyperbolic boundary is a doubling space, we conclude from~\cite[Corollary~10.2.4]{BS-Elements} that the hyperbolic space itself has finite asymptotic dimension.
We refer to~\cite{BS-Elements} for a broader overview of the distinct dimension concepts for hyperbolic spaces and to~\cite{BellDrani-Asymptotic, Gromov-Asymptotic} for more about the asymptotic dimension.

\section{Construction of the $\real$-tree}\label{Construction}\label{sec_Construction}

In this section we construct a rooted $\real$-tree $T$ inside a geodetic proper hyperbolic space $X$ whose hyperbolic boundary has finite Assouad dimension and whose hyperbolic constant is not~$0$.

The idea of the proof is similar to the proof of the main result in~\cite{HyperbolicSpanningTrees} for locally finite hyperbolic graphs.
But because the construction differs from the one in~\cite{HyperbolicSpanningTrees} and we are dealing with proper hyperbolic geodetic spaces instead of locally finite hyperbolic graphs, we prove the properties at the end of this section.

Let $d_h=d_\varepsilon$ be a metric such that $\varepsilon$ satisfies the assumptions as in Theorem~\ref{BoundComp} and hence such that $(\widehat{X},d_h)$ is a compact metric space.
By~\cite[Sections 6 and 9]{BonkSchramm_Embeddings} the property of $X$ to have finite Assouad dimension does not depend on the particular choice of~$\varepsilon$.
That means if $\rand X$ has finite Assouad dimension for one metric $d_\varepsilon$, then this holds for all these metrics.
That is the reason why we are able just to say that $\rand X$ has finite Assouad dimension.
By Theorem~\ref{AssouadDoubling}, $\rand X$ is a doubling metric space.
So let $N=2^{\dim_2(\rand X)}$.

The rooted $\real$-tree $T$ that we shall construct will have the following properties.
\begin{enumerate}[(1)]
\item Every ray in~$T$ converges to a point in the hyperbolic boundary of~$X$;
\item for every boundary point $\eta$ of~$X$ there is a ray in~$T$ converging to~$\eta$;
\item for every boundary point $\eta$ of~$X$ there are at most $N^{2+\log_2(8N^2)}$ distinct rays in~$T$ that start at the root of~$T$ and converge to~$\eta$.
\end{enumerate}

\medskip

We construct the rooted $\real$-tree $T$ recursively, so let $r\in X$ be the base-point of the Gromov-product which we used for the definition of the metric $d_\varepsilon$.
Then $r$ will be the root of~$T$.
For the construction of $T$ we construct a strictly descending sequence $\seq{\varepsilon}{j}{\nat}$ in $\real_{>0}$, two sequences $\seq{S}{j}{\nat},\seq{Y}{j}{\nat}$ of subsets of~$\rand X$, a sequence $\seq{\UF}{j}{\nat}$ of open covers of~$\rand X$, a sequence $\seq{\BF}{j}{\nat}$ of closed covers of~$\rand X$, and a sequence $\seq{T}{j}{\nat}$ of $\real$-trees that lie in $X$.
Our final tree $T$ will be the union of all the $T_j$.
The other sequences will help us in the construction of the $\real$-trees $T_j$ and they will satisfy the assertions (a) to~(k) for every~$j$.
\begin{enumerate}[(a)]
\item\label{item_epsilon} $\varepsilon_j=\frac{a}{8N^2}$ with $a=\frac{\varepsilon_{j-1}}{4\cdot16}$, so $\varepsilon_j=\frac{\varepsilon_{j-1}}{512N^2}$;
\item\label{item_SYS} there is $S_{j-1}\subseteq Y_j\subseteq S_j$;
\item\label{item_SDistance} $d_h(\eta,\mu)\ge \varepsilon_j$ for all $\eta\neq\mu\in S_j$;
\item\label{item_SMultiplicity} the set $S_j$ has $\frac{\varepsilon_{j-1}}{16}$-multiplicity at most $N^{\log_2(8N)}$;
\item\label{item_YDistance} $d_h(\eta,\mu)\ge \frac{\varepsilon_{j-1}}{4\cdot 16}$ for all $\eta\neq\mu\in Y_j$;
\item\label{item_UOpenBallAroundS} The open cover $\UF_j$ consists of precisely the open $\varepsilon_j$-balls around the elements of~$S_j$;
\item\label{item_B} the set $\BF_j$ consists of all closed balls of radius $\frac{\varepsilon_{j-1}}{4\cdot16}$ around the elements of~$Y_j$ and it has $\frac{\varepsilon_{j-1}}{4\cdot16}$-multiplicity at most $N^2$;
\item\label{item_YMulti} every $\eta\in Y_j$ has $(3\cdot\frac{\varepsilon_{j-1}}{4\cdot16})$-multiplicity at most $N^2$ in $\BF_j$;
\item\label{item_TT} $T_{j-1}\subseteq T_j$;
\item\label{item_RaysInTToS} every ray in $T_j$ converges to an elements of~$S_j$ and to each element converges precisely one ray that starts at the root;
\item\label{item_RaysInTGeod} every ray in $T_j$ is eventually geodetic, in particular, there is a constant $c$ depending only on $\varepsilon_j$ such that every ray in $T_j\sm \bar{B}_c(x)$ is a geodetic ray.
\end{enumerate}

Before we start the recursion step, we first define the elements of all sequences for $j=0$. 
Let $\mu^0\in\rand X$, $S_0=Y_0=\{\mu^0\}$ and $\varepsilon_0=\sup\{d_h(\mu^0,\eta)\mid \eta\in\rand X\}$---recall that $\rand X$ is bounded by Theorem~\ref{BoundComp}.
Let $\BF_0=\UF_0=\rand X$ and let $T_0$ be a geodetic ray from $r$ to~$\mu^0$ which exists by Proposition~\ref{prop_geodExists}.
Then all properties are satisfied for $j=0$.

For the recursion step we may choose the $\epsilon_j$ so that (\ref{item_epsilon}) holds.
Lemma~\ref{LS2.3} shows that there is a closed covering $\BF_j$ of $\rand X$ with balls of radius $\frac{\varepsilon_{j-1}}{4\cdot16}$ such that this covering has $\frac{\varepsilon_{j-1}}{4\cdot16}$-multiplicity at most $N^2$ and such that the set $Y_j$ of centers of these balls contains $S_{j-1}$ and such that every $\eta\in Y_j$ has $(3\cdot\frac{\varepsilon_{j-1}}{4\cdot 16})$-multiplicity at most $N^2$ in $\BF_j$.
Then (\ref{item_YDistance}), (\ref{item_B}), (\ref{item_YMulti}), and the first part of (\ref{item_SYS}) hold.

Let $S_j$ be a subset of $\rand X$ with $Y_j\subseteq S_j$ such that $d_h(\mu,\nu)\geq \varepsilon_j$ for all $\mu,\nu\in S_j$, such that $S_j$ has $\frac{\varepsilon_{j-1}}{16}$-multiplicity at most $N^{\log_2(8N)}$, and such that $\UF_j:=\{B_{\varepsilon_j}(\mu)\mid \mu\in S_j\}$ is an open cover of $\rand X$.
This set $S_j$ exists by applying the doubling definition $\log_N(N^{\log_2(8N)})=\log_2(8N)$ times to the sets in $\BF_j$ and it is finite because $\rand X$ is doubling.
As a consequence we have (\ref{item_SDistance}), (\ref{item_SMultiplicity}), (\ref{item_UOpenBallAroundS}), and the remaining part of (\ref{item_SYS}).
The only element of any of the sequences that remains to be constructed is the $\real$-tree $T_j$.

We construct the $\real$-tree $T_j$ by recursion.
Let $T^{0,1}_j=T_{j-1}$.
We enumerate the set $S_j\sm S_{j-1}$ in the following way.
Let $\mu_1^1,\mu_2^1,\ldots$ be the elements with $8\varepsilon_{{j-1}}$-multiplicity $1$ in $\BF_{{j-1}}$, let $\mu_1^2,\mu_2^2,\ldots$ be the elements with $(2\cdot8\varepsilon_{{j-1}})$-multiplicity at most $2$ in $\BF_{{j-1}}$, and so on.
As the set $\BF_{j-1}$ has $\frac{\varepsilon_{j-2}}{4\cdot 16}$-multiplicity at most $N^2$ and $8N^2\varepsilon_{j-1}=\frac{\varepsilon_{j-2}}{4\cdot16}\leq\rad(\BF_{j-1})$, there are no $\mu^i_k$ with $i>N^2$ by (\ref{item_B}).

The $\real$-tree $T^{i,k}_j$ shall be the union of the $\real$-tree $T^{i-1,k}_j$ and an eventually geodetic ray from $T^{i-1,k}_j$ to the hyperbolic boundary point $\mu^i_k$, where we denote with $T^{0,k}_j$ the union of all $T^{a,k-1}_j$.
So let $\mu^i_k\in S_j\sm S_{j-1}$ and assume that we have already constructed the $\real$-tree $T^{i-1,k}_j$.
There is a $\mu\in S_{j-1}$ with $d_h(\mu^i_k,\mu)\le \varepsilon_{j-1}$.
Let $R$ be a geodetic double ray from $\mu^i_k$ to~$\mu$.
Let $Q$ denote the largest distance from $r$ to any geodetic double ray between two boundary points of distance at most $\varepsilon_{j-1}$ and at least $\varepsilon_j$ and let $q$ denote the smallest distance of from $r$ to such a double ray.
Then we know that there is $\beta=Q-q$ for the constant $\beta$ from Proposition~\ref{prop_QuotientLeadsToDistance}.

Let us first consider the case that there is a common point of~$R$ and $T^{i-1,k}_j$ that has distance at most $Q+5\delta$ to~$r$.
Then there is a first common point $x$ of $R$ and $T^{i-1,k}_j$ such that $Rx$, the ray from $x$ to~$\mu^i_k$, contains no other point of~$T_j^{i-1,k}$ by the compactness of the ball of radius $Q+5\delta$.
In this case we just add the subray $Rx$ to the boundary point $\mu^i_k$ to the $\real$-tree to obtain the $\real$-tree $T^{i,k}_j$.
By the choice of~$x$, $T^{i,k}_j$ is indeed an $\real$-tree.
In preparation of the proof of Lemma~\ref{lem_TreeIsQuasiGeo} we denote with $\pi_R$ the point $x$ and for Claim~\ref{clm_1stClaim} we set $x_P:=x$.
If $x$ lies during the construction on a geodetic double ray $P$, then we say that we have {\em connected} $\mu^i_k$ to that limit point $\eta$ of~$P$ that has smaller distance to~$\mu^i_k$.
If $x$ lies on some $\pi_P$ for a double ray $P$, then we have {\em connected} $\mu^i_k$ either to the boundary point $\eta$ we constructed a new ray to with $P$ or inductively to one of the possible boundary points we connected $\eta$ to, depending which one has the smallest distance to $\mu^i_k$.
Since the hyperbolic boundary has the doubling property, the described relation is well-defined.
If the hyperbolic boundary point $\eta$ to which $\mu^i_k$ is connected lies in $S_{j-1}$, then $\mu^i_k$ is {\em eventually connected} to~$\eta$.
If this is not the case, then $\mu^i_k$ is {\em eventually connected} to that hyperbolic boundary point to which $\eta$ is eventually connected.

Now we look at the case that $R$ is totally distinct from $T^{i-1,k}_j$ in the ball with center $r$ and radius $Q+5\delta$.
There is a geodetic ray $P$ in $T^{i-1,k}_j$ converging to~$\mu$ whose first point has distance $Q+5\delta$ to~$r$.
Let $x_P$ be the starting point of~$P$.
Then $d(o,x_P)\leq Q+5\delta$.
We consider a geodetic path $\tilde{\pi}_R$ from $R$ to~$P$ that has length at most~$\delta$ with the additional property that for a point $z$ on~$R\cap B_Q(r)$ we have that $d_{R\cup \pi_R}(z,x_P)$ is minimal.
This exists because every point $y$ on~$P$ with $d(r,y)=Q+3\delta$ is $\delta$-close to a point on~$R$ and because $X$ is proper.
As it lies in the ball with center $r$ and radius $Q+6\delta$, there is, by the compactness of that ball, a smallest connected subpath $\pi_R$ of~$\tilde{\pi}_R$ that contains a point of~$R$ and a point of~$T^{i-1,k}_j$.
Let $x_R$ denote the intersection point of $\pi_R$ and $R$.
Then we add the subray of~$R$ from $\pi_R$ to $\mu^i_k$ together with $\pi_R$ to $T^{i-1,k}_j$ to obtain a new $\real$-tree $T^{i,k}_j$.
The property for $\mu^i_k$ of being {\em connected} is defined analog to the first case.

Let $T_j:=\bigcup_{i,k} T^{i,k}_j$.
We shall show that $T_j$ is a $\real$-tree again.
But this is an easy observation because $T_j$ is the union of a chain of finitely many $\real$-trees.
We remark that the $\real$-tree $T_j$ satisfies the properties (\ref{item_TT}), (\ref{item_RaysInTToS}), and (\ref{item_RaysInTGeod}) for $c=Q$.

\medskip

We just have defined all sequences as claimed.
Set$$T:=\bigcup_{j\in\nat}T_j.$$
It remains to prove that $T$ is a $\real$-tree and satisfies the assertions (1) to (3) as claimed.
As each of the $\real$-trees $T_j$ is connected and $T_j\sub T_{j+1}$, we know that $T$ is connected.
So we just have to show that $T$ does not contain any circle.

We construct an auxiliary graph: For every $j\in\nat$ let $\TF_j$ be the graph whose vertex set consists of~$r$, of all ends of~$T_j$, and of all points $x$ of~$T_j$ for which there are three non-trivial and pairwise disjoint (except for~$x$) paths.
Two vertices are adjacent if and only if there is a geodesic in $T_j$ between these two points that contains no other vertex of~$\TF_j$.
There is a canonical injective homomorphism\footnote{A {\em homomorphism} from a graph $G$ to a graph $H$ is a map from the vertex set of~$G$ to the vertex set of~$H$ such that the images of each two adjacent vertices are adjacent in~$H$.} $\gamma_j$ from $\TF_j$ to~$\TF_{j+1}$ that is the identity on $V\TF_j\cap T_j$ but that may map ends of~$T_j$ to points of~$T_{j+1}$.
By the construction of~$T_j$, the graphs $\TF_j$ are finite trees.
Let $\TF:=\bigcup_{j\in\nat}\TF_j$.
We constructed the trees $T_j$ so that for every $\varrho>0$ there is a $j_\varrho$ such that for every $j\geq j_\varrho$ we have $T_j\cap \bar{B}_\varrho(r)=T_{j_\varrho}\cap \bar{B}_\varrho(r)$.
Hence for every vertex $x$ of~$\TF$ there is an index $j_x$ such that the degree of~$x$ in~$\TF$ is the same as the one in~$\TF_{j_x}$.
We first show that $\TF$ is a tree.
If this is not the case, then there is a {\em cycle} $C$ in~$\TF$ that is a finite sequence of vertices such that any element of that sequence is adjacent to its successor and the last and the first vertex are also adjacent.
Since $C$ has only finitely many edges, there is a $j\in\nat$ such that $C$ is contained in~$\TF_j$ which is impossible because all the $\TF_j$ are trees.
Next we observe that---if we consider the graphs to be $1$-complexes---we can define homeomorphisms $\tau_j$ from $\widehat{T}_j$ to~$\TF_j$ that are the identity on $V_j:=V\TF_j\cap T_j$ and with $\tau_j|_{\TF_j[V_j]}=\tau_{j+1}|_{\TF_j[V_j]}$\footnote{For a graph $G$ and a subset $S$ of its vertices, $G[S]$ denotes the graph with vertex set $S$ and an edge between $x,y\in S$ if and only if it is an edge in~$G$.}.
These homeomorphisms converge to a homeomorphism $\tau:\widehat{T}\to\widehat{\TF}$.
But then for any circle in~$T$ its image in~$\TF$ is contained in a cycle in~$\TF$.
As there are no cycles in~$\TF$, we also have no circles in~$T$, so $T$ is an $\real$-tree.

\bigskip

In order to prove the assertions (1) to (3) we shall prove several claims in which we use the notation from the construction step $(j,k,i)$.

\begin{clm}\label{clm_1stClaim}
There is $d_{T^{i,k}_j}(R,x_P)\leq \delta$.
\end{clm}

\begin{proof}[Proof of Claim~\ref{clm_1stClaim}]
By induction, we know that the corresponding statement holds for all previous $\real$-trees.
If $\pi_R$ does not meet $T^{i-1,k}_j$ except for $x_P$, then the assertion holds trivially, so we may assume that $\pi_R$ meets some other $R'$ or $\pi_{R'}$ (these correspond to~$R$ or $\pi_R$ in a previous step).
Suppose first, that it meets some $\pi_{R'}$.
Then this $\pi_{R'}$ has to have distance at most $\delta$ to~$x_P$, because otherwise the corresponding point $x_{P'}$ lies at distance at most $2\delta$ from~$x_P$ and thus for the two hyperbolic boundary points to which $P$ and $P'$ converge, $\xi$ and $\xi'$, respectively, any geodetic double ray between them lies in a $3\delta$-neighborhood of $P\cup \pi_R\cup\pi_{R'}\cup P'$, so at least $Q+\delta$ away from~$r$ and hence we have $d_h(\xi,\xi')<\varepsilon_{j}$ which is impossible as soon as $\xi\neq\xi'$.
Let us now suppose that $\pi_R$ meets some $R'$.
Then we have chosen $\pi_{R'}$ so that every point on it has distance at most $\delta$ to~$x_P$ in the $\real$-tree by the same contradiction as above.
By the minimality of~$d_{R'\cup\pi_{R'}}(z',x_P)$ for the point $z'$ that corresponds to~$z$ for $R'$ instead of~$R$, we know that the claim must hold.
\end{proof}

\begin{clm}\label{clm_2ndClaim}
Let $\mu^n_{k}$ and $\mu^n_{l}$ be two elements of $S_j\setminus S_{j-1}$ with $d_h(\mu^n_{k},\mu^n_{l})\leq 8\varepsilon_{j-1}$ for an $n\leq N^2$.
Then for any $B\in\BF_{j-1}$ with $d_h(\mu^n_{k},B)\leq n8\varepsilon_{j-1}$ there is $d_h(\mu^n_{l},B)\leq n8\varepsilon_{j-1}$.
\end{clm}

\begin{proof}[Proof of Claim~\ref{clm_2ndClaim}]
The $((n-1)8\varepsilon_{j-1})$-multiplicity of~$\mu^n_{k}$ and the one of~$\mu^n_{l}$ in~$\BF_{j-1}$ has to be~$n$.
Thus for every hyperbolic boundary point $\eta$ in~$Y_{j-1}$ with distance at most $n8\varepsilon_{j-1}$ to $\mu^n_{k}$ we have $d_h(\eta,\mu^n_k)\leq (n-1)8\varepsilon_{j-1}$ and hence also $d_h(\eta,\mu^n_l)\leq n8\varepsilon_{j-1}$.
\end{proof}

\begin{clm}\label{clm_3rdClaim}
Let  $\mu^k_{i+1}$ be connected to~$\mu\in S_j$. Then $d_h(\mu,\mu^k_{i+1})\leq 8\varepsilon_{j-1}$.
If $\mu^k_{i+1}$ is eventually connected to $\eta\in S_{j-1}$ in~$T_j$, then$$d_h(\eta,\mu^k_{i+1})\leq 8N^2\varepsilon_{j-1}+\rad(\BF_{j-1})=16N^2\varepsilon_{j-1}.$$
Furthermore, $\eta$ lies in some $B\in\BF_{j-1}$ with $d_h(\mu^k_{i+1},B)\leq 8N^2\varepsilon_{j-1}$.
\end{clm}

\begin{proof}[Proof of Claim~\ref{clm_3rdClaim}]
Let us first prove $d_h(\mu^k_{i+1},\mu)\leq 8\varepsilon_{j-1}$.
An immediate consequence of Claim~\ref{clm_1stClaim} is, if we inserted a geodetic part $\pi_R$, that then the boundary point we connected $\mu^k_{i+1}$ to has distance at most $\varepsilon_{j-1}$ to~$\mu^k_{i+1}$.

So we assume in the following that $\pi_R$ is only one point.
Then $R$ meets some other double ray $R'$ or a geodetic segment $\pi_{R'}$ where $R'$ and $\pi_{R'}$ are as in the proof of Claim~\ref{clm_1stClaim}.
If $R$ meets some other double ray $R'$, then $\mu$ is the limit point of~$R'$ and any geodetic double ray $[\mu^k_{i+1},\mu]$ lies in a $\delta$-neighborhood of~$R\cup R'$, so it has distance at least $q-\delta$ to~$r$.
Then we have with
$$\delta':=\exp(6\varepsilon\delta)\leq(\sqrt{2})^6=8$$
for any $\mu'\in\rand X$ with $\varepsilon_j\leq d_h(\mu^k_{i+1},\mu')\leq\varepsilon_{j-1}$
$$\begin{array}{lll}
d_h(\mu^i_k,\mu)&\le&\exp(-\varepsilon(\mu^i_k,\mu))\\
&\leq&\exp(-\varepsilon(\mu^i_k,\mu')+5\varepsilon\delta)\\
&\leq&\frac{\delta'}{\varepsilon'}\exp(-\varepsilon(\mu^i_k,\mu'))\\
&\leq&8d_h(\mu^i_k,\mu')\\
&\leq&8\varepsilon_{j-1}.
\end{array}$$
Now we assume the last case, that is, $R$ meets some $\pi_{R'}$.
Then any point on~$R\cap\pi_{R'}$ has distance at most $\delta$ to~$x_{P'}$, where $x_{P'}$ denotes the point for $R'$ that $x_P$ denotes for~$R$.
Let $P'$ be the ray for~$R'$ that is $P$ for~$R$.
We conclude that there is a hyperbolic boundary point $\mu'$ such that $[\mu^k_{i+1},\mu']$ lies in a $2\delta$-neighborhood of $R\cup\pi_{R'}\cup P'$.
This gives us the following inequality for any $\nu\in\rand X$ with $\varepsilon_j\leq d_h(\mu^k_{i+1},\nu)\leq\varepsilon_{j-1}$.
$$\begin{array}{lll}
d_h(\mu^i_k,\mu)&\le&\exp(-\varepsilon(\mu^i_k,\mu))\\
&\leq&\exp(-\varepsilon(\mu^i_k,\nu)+5\varepsilon\delta)\\
&\leq&\frac{\delta'}{\varepsilon'}\exp(-\varepsilon(\mu^i_k,\nu))\\
&\leq&8d_h(\mu^i_k,\nu)\\
&\leq&8\varepsilon_{j-1}.
\end{array}$$

Let $m$ be minimal such that the $((m-1)\cdot 8\varepsilon_{j-1})$-multiplicity of~$\mu$ is not $m-1$ but such that the $(m\cdot 8\varepsilon_{j-1})$-multiplicity of~$\mu$ is~$m$.
If $m=1$, then the two boundary points $\mu^k_{i+1}$ and~$\mu$ lie in the same ball $B\in\BF_{j-1}$.
We conclude that all three hyperbolic boundary points $\mu^k_{i+1}$, $\mu$, and $\eta$ lie in a common ball $B\in\BF_{j-1}$ and hence that$$d_h(\eta,\mu^k_{i+1}) \leq \rad(\BF_{j-1})\leq 8N^2\varepsilon_{j-1}.$$
Let us now assume that $m\neq 1$.
We may assume that $\mu\neq\eta$, that is $\mu\in S_j\setminus S_{j-1}$.
By induction we know that $\eta$ lies in one of the elements of~$\BF_{j-1}$, say in $B$, that is responsible for the $(n\cdot8\varepsilon_{j-1})$-multiplicity of at most $n$ of~$\mu$ where $n$ denotes the corresponding value for~$\mu$ that is $m$ for~$\mu^k_{i+1}$.
As~$\mu^k_{i+1}$  is connected to~$\mu$, we have $n\leq m$.
Thus $d_h(\mu^k_{i+1},B)\leq m\cdot 8\varepsilon_{j-1}$ and hence there is
$$d_h(\mu^k_{i+1},\eta)\leq m\cdot8\varepsilon_{j-1}+\rad(\BF_{j-1})\leq m\cdot8\varepsilon_{j-1}+8N^2\varepsilon_{j-1}.$$
Since every element of $S_j\setminus S_{j-1}$ has $(8N^2\varepsilon_{j-1})$-multiplicity at most $N^2$ in~$\BF_{j-1}$, we have $d_h(\mu^k_{i+1},\eta)\leq 16N^2\varepsilon_{j-1}$.
\end{proof}

By the construction of the $\real$-trees $T_j$, we have the following property.
\begin{itemize}
\item[($*$)] In every step and for every closed ball $B\in\BF_k$ a boundary point in~$B$ can only be eventually connected to elements of at most $N^2$ different balls in~$\BF_k$.
Furthermore, there are at most $N^{\log_2(8N^2)}$ distinct boundary points in~$B$ that are eventually connected to elements of the same ball of~$\BF_k$.
\end{itemize}

Now we are ready to prove the assertions (1) to (3) for the $\real$-tree~$T$.
For a closed ball $B\in\BF_k$ let $B'$ be the union of~$B$ and all other (at most~$N^2$) closed balls in~$\BF_k$ with distance at most $8N^2\varepsilon_k$ to~$B$.

Because of~(\ref{item_RaysInTToS}) we just have prove that any ray that we created without our knowledge in the limit step converges to some hyperbolic boundary point.
Let $\pi$ be such a ray in~$T$.
We remark that we do not need any of the properties (1) to (3) for the proof of Lemma~\ref{lem_TreeIsQuasiGeo} which we apply at this place.
The lemma says that $\pi$ is eventually a quasi-geodetic ray.
We deduce from Proposition~\ref{prop_geodAndQuasigeod} that there is a geodetic ray~$\widehat{\pi}$ and a $\kappa\geq 0$ such that $\pi$ lies in a $\kappa$-neighborhood of~$\widehat{\pi}$.
Thus, $\pi$ converges to the same boundary as $\widehat{\pi}$ and we have proved (1).

\medskip

For the proof of~(2), let $\eta\in\rand X$.
In every construction step $k$ there is at least one closed ball $B_k\in\BF_k$ with $\eta\in B_k$ because $\BF_k$ is a cover of~$\rand X$.
Hence there is in each step a boundary point $\eta_k\in S_k\cap B_k$ with $d_h(\eta_k,\eta)\leq\varepsilon_k$ such that $T_k$ contains a ray to~$\eta_k$.
Let $\pi_k$ be a ray from $r$ to~$\eta_k$ in~$T_k$.
For every $\varrho\in\nat$ there is a path in $T_k\cap \bar{B}_\varrho(r)$ that is contained in infinitely many of the~$\pi_k$ by the compactness of~$\bar{B}_\varrho(r)$ and because there are only finitely many paths in~$T_j$ that starts at $r$ and end at a point with distance $\varrho$ from~$r$.
Thus there is a ray $\pi$ such that every point on~$\pi$ lies on infinitely many of the rays~$\pi_k$.
Because of Claim~\ref{clm_3rdClaim} and the choice of the rays $\pi_k$, the hyperbolic boundary point $\eta$ has to be an accumulation point of~$\pi$.
As (1) holds, $\pi$ has precisely one accumulation point, $\eta$, and thus $\pi$ converges to~$\eta$.

\medskip

For every $B\in\BF_k$ in the step $k$ there are at most $N^2$ closed balls in the step $k-1$ such that a boundary point in~$(B\cap S_k)\sm S_{k-1}$ is eventually connected to a hyperbolic boundary point of such a ball.
Furthermore, for each of these balls there are at most $N^{\log_2(8N^2)}$ many hyperbolic boundary points to which our new ones are eventually connected.
Thus we know that the number of rays to one boundary point is bounded by $N^2\cdot N^{\log_2(8N^2)}$ and hence bounded by a function depending only on $\dim_2(\rand X)$.
Thus, we have also proved the remaining assertion (3).

\section{Visual hyperbolic spaces}\label{sec_DefVisual}

As in \cite{BDS-Embedding}, we call a hyperbolic space $X$ {\em visual} if for some $o\in X$ there is a $D>0$ such that for every $x\in X$ there is an $\eta\in\rand X$ with$$d(o,x)\leq (x,\eta)_o+D.$$
Remark that the property for hyperbolic spaces to be visual is independent of the choice of $o$.

An observation is that the definition of visual hyperbolic spaces is equivalent to the following.
For some (and hence every) $o\in X$ there is a $D'>0$ such that for every $x\in X$ there is an $\eta\in\rand X$ such that any geodesic between $o$ and $x$ lies in a $D'$-neighborhood of a geodetic ray from $o$ to~$\eta$.

Remark that by Corollary~1.3.5.\ of \cite{BS-Elements} hyperbolicity is preserved by quasi-isometries and it is not hard to see that the same holds for visual hyperbolicity.

\subsection{Hyperbolic approximations of metric spaces}\label{sec_subsecApproximation}

In~\cite{BS-Elements}, see also \cite{BourdonPajot, BDS-Embedding,Elek-Cohomology}, the authors construct for every metric space $X$ a hyperbolic space $Y$ whose hyperbolic boundary is homeomorphic to~$X$.
The hyperbolic space $Y$ is called a {\em hyperbolic approximation of~$X$}.
That $Y$ is indeed a hyperbolic space is shown in~\cite[Proposition~6.2.10]{BS-Elements} and we just state the proposition without proof.
Remark that it is easy by looking at the construction to see that $Y$ is visual hyperbolic since any vertex of~$Y$ lies on an infinite geodetic ray that starts at the {\em root} of the hyperbolic approximation.

\begin{prop}\label{prop_BS6.2.10}{\em \cite[Proposition~6.2.10]{BS-Elements}}
A hyperbolic approximation $Y$ of any metric space $X$ is a visual hyperbolic graph with $\rand Y\isom X$.\qed
\end{prop}

If we restrict the metric space $X$ to be doubling, then the degrees of all the vertices in a hyperbolic approximation of~$X$ are uniformly bounded by \cite[Proposition 8.3.3]{BS-Elements}.
We combine this result with Proposition~\ref{prop_BS6.2.10} and obtain the following proposition.

\begin{prop}\label{prop_BS6.2.10Cor}
A hyperbolic approximation $Y$ of any doubling metric space $X$ is a visual hyperbolic locally finite graph with $\rand Y\isom X$ and with bounded degree.\qed
\end{prop}

\subsection{Rough similarities}

We cite a result by Buyalo and Schroeder \cite{BS-Elements}.
In order to do that we have to make a further definition.

Let $X,Y$ be two metric spaces. If there are a map $f:X\to Y$ and constants $k,\lambda>0$ such that
$$|\lambda d_X(x,y)-d_Y(f(x),f(y))|\leq k$$
holds for all $x,y\in X$ and $\sup_{y\in Y}d_Y(y,f(X))\leq k$, then $X$ is {\em $(\lambda,k)$-roughly similar} to~$Y$, or just {\em roughly similar} to~$Y$, and we call $f$ a {\em $(\lambda,k)$-rough similarity}, or just a {\em rough similarity}.

In particular, every space $Y$ that is roughly similar to a space $X$ is also quasi-isometric to $X$.
As (visual) hyperbolicity is preserved by quasi-isometries, it is also preserved by rough similarities.

\begin{thm}{\em \cite[Corollary 7.1.5.]{BS-Elements}}\label{thm_BuSCor7.1.5}
Every visual hyperbolic space $X$ is roughly similar to a subspace of a hyperbolic geodesic space $Y$ with the same hyperbolic boundary, $\rand X=\rand Y$.\qed
\end{thm}

We conclude the following corollary from the previous theorem.

\begin{cor}\label{cor_BuSCor7.1.5}
Let $X$ be a proper hyperbolic geodetic space whose hyperbolic boundary is doubling.
Let $\gamma_1\geq 1,\gamma_2\geq 0$ be constants.
Then there is a subspace $Y$ of~$X$ such that the following statements hold for~$Y$.
\begin{enumerate}[$(1)$]
\item $Y$ is a proper visual hyperbolic geodetic space;
\item every $(\gamma_1,\gamma_2)$-quasi-geodetic ray of~$X$ lies eventually in~$Y$;
\item the identity $\iota:Y\to X$ extends to a homeomorphism $\hat{\iota}:\widehat{Y}\to\widehat{X}$ such that $\hat{\iota}(\rand Y)=\rand X$.
\end{enumerate}
\end{cor}

\begin{proof}
Let $Z$ be a visual hyperbolic locally finite graph that is a hyperbolic approximation of the hyperbolic boundary $\rand X$.
Let $Z'$ be a subspace of~$X$ that is $(\lambda,k)$-roughly similar to~$Z$ for some constants $\lambda\geq 1,k\geq 0$ by Theorem~\ref{thm_BuSCor7.1.5}, and let $Y$ be the subspace of~$X$ that is induced by~$Z'$ and all points with distance at most $\kappa(\delta,\gamma_1,\gamma_2)+2\,\kappa(\delta,\lambda,k)$ to any element of~$Z'$ for the constants $\kappa(\delta,\gamma_1,\gamma_2),\kappa(\delta,\lambda,k)$ of Proposition~\ref{prop_geodAndQuasigeod}.
Since $Z$ is locally finite and $X$ is proper, the space $Z'$ is proper and the same holds for~$Y$.
As $Z$ is visual hyperbolic and this is a property that is preserved by quasi-isometries, assertion (1) holds for $Z'$ and thus also for $Y$ as the identity from $Z'$ to~$Y$ is a quasi-isometry by the choice of~$Y$.
The assertion (2) holds because of Proposition~\ref{prop_geodAndQuasigeod} since $Z$ is a geodetic space, and hence every two points of~$Z'$ can be joined by a $(\lambda,k)$-quasi-geodesic.
Finally, the assertion (3) is obvious, because quasi-isometries between proper hyperbolic geodetic spaces can be extended to quasi-isometries between their hyperbolic compactifications.
\end{proof}

\section{Tree-likeness of hyperbolic spaces}

We remark that the constructed tree in~\cite{HyperbolicSpanningTrees} for locally finite hyperbolic graphs is in general far from having only rays that are eventually quasi-geodetic.
But the changes in the construction we made in this paper are strong enough to guarantee that all rays in the constructed $\real$-tree are already eventually quasi-geodetic rays in the hyperbolic space.

\begin{lem}\label{lem_TreeIsQuasiGeo}
Let $X$ be a proper hyperbolic geodetic space whose hyperbolic boundary has finite Assouad dimension and let $T$ be the $\real$-tree that was constructed in Section~\ref{Construction} with root~$r$.
There exist constants $\gamma_1\geq 1,\gamma_2\ge 0$ such that every ray in~$T$ starting at the root is a $(\gamma_1,\gamma_2)$-quasi-geodetic ray in $X$.
\end{lem}

\begin{proof}
We assume all assumptions and notations as in the construction step $j$ of Section~\ref{Construction}.
By Proposition~\ref{prop_QuotientLeadsToDistance} there is a constant $\beta$ depending only on the quotient $\frac{\varepsilon_j}{\varepsilon_{j-1}}$ and not depending on the particular $j$ such that for every four boundary points $\eta_1,\eta_2,\eta_3,\eta_4$ with
$$\varepsilon_{j-1}\geq d_h(\eta_1,\eta_2),d_h(\eta_3,\eta_4)\geq\varepsilon_j$$
there is
$$|d(r,[\eta_1,\eta_2])-d(r,[\eta_3,\eta_4])|\leq\beta.$$
Recall that $\beta=Q-q$ with $Q,q$ as defined in Section~\ref{Construction}.
In the first step of the proof we shall prove that for every two points $w,y$ with $y\in T^{i,k}_j\setminus T^{i-1,k}_j$, $w\in T^{i,k}_j\cap [r,y]_T$ there is
$$d_Y(w,y)\leq d(w,y)+(M+n)(75\delta+4\beta)$$
with $Y:=T^{i,k}_j$ for an $n<M:=N^2N^{\log_2(8N^2)}$ that represents the number how often we have already enlarged the tree $T_{j-1}$ by additional rays whose intersection with $[r,y]_T$ is not empty.
As we have proved, $n$ is bounded by $M$ and since we add just in this step a ray, there is $n<M$.

So let $y\in Y$.
Let $R$ be the geodetic double ray, as in the recursion step and let $P$ be that geodetic ray with $P\subseteq R$ that we added together with $\pi_R$ to $T^{i-1,k}_j$ to obtain $T^{i,k}_j$.
Let $x$ be the unique point in $T^{i-1,k}_j\cap\pi_R$, let $x'$ be the unique point in $\pi_R\cap P$, and let $a$ be a point on $R$ with minimal distance to~$r$.
By the choice of $\pi_R$ we have $d(x,x')\leq\delta$ as we already saw in Section~\ref{Construction}.

Let $\eta:=\mu^i_k$ and let $\mu$ be the other limit point of~$R$. Let $b\in R$ with $d(b,[r,\eta])\leq\delta$ and $d(b,[r,\mu])\leq\delta$.

\begin{clm}\label{clm_TreeIsQuasiGeo1}
$d(a,b)\leq 4\delta$.
\end{clm}

\begin{proof}[Proof of Claim~\ref{clm_TreeIsQuasiGeo1}]
Let $c\in[r,\eta]$ and $c'\in[r,\mu]$ each with minimal distance to~$b$.
By the choice of~$b$ there is $d(b,c),d(b,c')\leq\delta$.

By the hyperbolicity of~$X$ the geodetic double ray $R$ is contained in the $\delta$-neighborhood of the subset $Z:=[\eta,c]\cup[c,b]\cup[b,c']\cup[c',\mu]$ of~$X$.
In particular we have $d(a,Z)\leq\delta$.
Thus there is a point on $Z$ with distance at most $d(r,a)+\delta$ to~$r$.
Let $a'\in Z$ with $d(a,a')\leq\delta$.
Then there is $d(r,a')\leq d(r,a)+\delta$.
By symmetry we may assume that $a'\in[\eta,c]\cup[c,b]$.
If $a'\in [c,b]$, then we have $d(a',c)\leq \delta$.
Otherwise, since $c$ is the point on $[r,\nu]\cap Z$ with minimal distance to~$r$, and since $d(r,c),d(r,c')\geq d(r,a)-\delta$, we have $d(a',c)\leq 2\delta$.
The inequality
$$d(a,b)\leq d(a,a')+d(a',c)+d(c,b)\leq\delta+2\delta+\delta=4\delta$$
proves Claim~\ref{clm_TreeIsQuasiGeo1}.
\end{proof}

For any another point $\widehat{a}$ on~$R$ with distance $d(r,a)$ to~$r$ we conclude from Claim~\ref{clm_TreeIsQuasiGeo1} that $d(a,\widehat{a})\leq 8\delta$.

\begin{clm}\label{clm_TreeIsQuasiGeo2}
$d(a,x')\leq\beta+15\delta$.
\end{clm}

\begin{proof}[Proof of Claim~\ref{clm_TreeIsQuasiGeo2}]
Let $x''$ be the point on $[r,x']$ with $d(r,x'')=d(r,a)-\delta$. In particular we have $d(x',x'')\leq\beta+6\delta$.
Since $X$ is hyperbolic, there is a point on $[r,a]\cup[a,x']$ with distance at most $\delta$ to $x''$.
If this point lies on $[r,a]$, then $d(x'',a)\leq 3\delta$, and if this point lies on $[a,x']$, then it has the same distance to~$r$ as~$a$ and hence distance at most $8\delta$ to~$a$.
Thus $d(x'',a)\leq 9\delta$.
Hence we proved Claim~\ref{clm_TreeIsQuasiGeo2}.
\end{proof}

Let $a_w$ be a point on~$R$ with $d(w,a_w)=d(w,R)$.

\begin{clm}\label{clm_TreeIsQuasiGeo3}
$d(w,y)\geq d(w,a_w)+d(a_w,y)-6\delta$.
\end{clm}

\begin{proof}[Proof of Claim~\ref{clm_TreeIsQuasiGeo3}]
Let $y'$ a point on $[w,y]$ with distance at most $\delta$ to both $[w,a_w]$ and $[a_w,y]$.
Let $y_1$ be such a point on $[w,a_w]$ and $y_2$ such a point on $[a_w,y]$.
Then $d(w,y_2)\geq d(w,a_w)$ and hence $d(y_1,a_w)\leq 2\delta$.
This immediately implies $d(w,y')\geq d(w,a_w)-3\delta$ and also $d(y,y')+3\delta\geq d(a_w,y)$.
Hence we proved Claim~\ref{clm_TreeIsQuasiGeo3}
\end{proof}

\begin{clm}\label{clm_TreeIsQuasiGeo4}
$d(a,a_w)\leq 19\delta+\beta$.
\end{clm}

\begin{proof}[Proof of Claim~\ref{clm_TreeIsQuasiGeo4}]
Since $X$ is hyperbolic we conclude directly, that $[a,a_w]$ lies in a $2\delta$-neighborhood of $[a,r]\cup[r,w]\cup[w,a_w]$.
But since $d(r,a)=d(r,R)$ and $d(w,a_w)=d(w,R)$, a part of length at most $4\delta$ of $[a,a_w]$ lies in the $2\delta$-neighborhood of $[r,a]$ and a part of length at most $4\delta$ lies in the $2\delta$-neighborhood of~$[w,a_w]$.
The point $w$ lies in $[r,y]\cap T^{i-1,k}_j$ and hence $d(r,w)\leq d(r,a)+\beta+5\delta$, so there is at most a part of length $11\delta+\beta$ of~$[a,a_w]$ in a $2\delta$-neighborhood of~$[r,w]$.
Then $d(a,a_w)$ is at most $19\delta+\beta$.
This proves Claim~\ref{clm_TreeIsQuasiGeo4}.
\end{proof}

Let $a'_w$ be a point on~$P$ with minimal distance to~$w$.
By analog arguments as in Claim~\ref{clm_TreeIsQuasiGeo2} there is $d(a_w,a'_w)\leq\beta+15\delta$.
Then we conclude
$$\begin{array}{lll}d_Y(w,y)&=&d_Y(w,a'_w)+d_Y(a'_w,y)\\
&\leq&d_Y(w,x)+d(x,x')+d(x',a)+d(a,a_w)+d(a_w,a'_w)\\
&&+d(a'_w,y)\\
&\leq&d_Y(w,x)+\delta+\beta+15\delta+\beta+19\delta+d(a_w,y)\\
&\leq&d_Y(w,x)+34\delta+2\beta+d(a_w,y)\\
&\leq&d(w,x)+((j-(j'+1))M+n)(\alpha_1\delta+\alpha_2\beta)\\
&&+34\delta+2\beta+d(a_w,y)\\
&\leq&d(w,a_w)+d(a_w,a)+d(a,x)+d(a_w,y)+34\delta+2\beta\\
&&((j-(j'+1))M+n)(\alpha_1\delta+\alpha_2\beta)\\
&\leq&d(w,a_w)+d(a_w,y)+19\delta+\beta+\beta+16\delta+34\delta+2\beta\\
&&((j-(j'+1))M+n)(\alpha_1\delta+\alpha_2\beta)\\
&\leq&d(w,y)+6\delta+69\delta+4\beta+((j-(j'+1))M+n)(\alpha_1\delta+\alpha_2\beta)\\
&\leq&d(w,y)+((j-(j'+1))M+(n+1))(\alpha_1\delta+\alpha_2\beta)
\end{array}$$
with $\alpha_1=75$ and $\alpha_2=4$.
And in particular we have
$$d_Y(w,y)\leq d(w,y)+((j-(j'+1))M+(n+1))(75\delta+4\beta).$$

Let $\pi$ be a ray in~$T$ that starts at~$r$.
Since every step affects at most its previous and its successive step directly, there are constants $c_1,c_2$ (independent from the choice of~$\pi$) such that $\pi$ is a $(c_1,c_2)$-quasi-geodetic ray.
\end{proof}

This lemma enables us to prove our main result. We will prove it in two steps.
First we prove the result for proper visual hyperbolic geodetic spaces and then for arbitrary proper hyperbolic geodetic spaces.

\subsection{The case: visual hyperbolic spaces}

Visual hyperbolic spaces seem to have a treelike-structure, since there is a maximal distance from each point to any ray starting at the same vertex.
This in fact is the main reason why the $\real$-tree constructed in Section~\ref{Construction} points out the tree-likeness of visual hyperbolic spaces.
This is specified in Theorem~\ref{thm_mainPart1}.

For a hyperbolic space $X$ and a subspace $T$ of~$X$, we say that the {\em canonical map from $\rand T$ to~$\rand X$ exists} if the identity $\iota:T\to X$ extends to a continuous map $\hat{\iota}:\widehat{T}\to\widehat{X}$ such that $\hat{\iota}(\rand T)=\rand X$ and $\hat{\iota}|_{\rand T}$ is the {\em canonical map} from $\rand T$ to~$\rand X$.

\begin{thm}\label{thm_mainPart1}
Let $X$ be a proper visual hyperbolic geodetic space whose hyperbolic boundary has finite Assouad dimension.
Then there is an $\real$-tree $T\sub X$ such that the canonical map $\gamma$ from $\rand T$ to~$\rand X$ exists and has the following properties.
\begin{enumerate}[{\em (i)}]
\item It is surjective;
\item there is a constant $M<\infty$ depending only on the Assouad dimension of $\rand X$ such that $\gamma\inv(\eta)$ has at most $M$ elements for each $\eta\in\rand X$;
\item there is a constant $\Delta<\infty$ depending only on~$\delta$ and on the Assouad dimension of~$\rand X$ such that every point of~$X$ lies in a $\Delta$-neighborhood of a point of~$T$.
\end{enumerate}
\end{thm}

\begin{proof}
Let $T$ be the $\real$-tree constructed in Section~\ref{sec_Construction} with root~$r$.
We already proved in that section the properties (i) and (ii). For the remaining one we remember from Lemma~\ref{lem_TreeIsQuasiGeo} that there exist constants $c_1,c_2$ such that each ray in~$T$ that starts at the root is a $(c_1,c_2)$-quasi-geodetic ray.
Because $X$ is visual hyperbolic, there is a $D>0$ such that for every $x\in X$ there is an $\eta\in\rand X$ with $d(x,\pi)\leq D$ for all geodetic rays $\pi$ from $r$ to~$\eta$.
Let $\pi_x$ be a point on~$\pi$ with $d(x,\pi_x)\leq D$.
In~$T$ there is a ray $\pi_T$ from~$r$ converging to~$\eta$.
We know from Proposition~\ref{prop_geodAndQuasigeod} that there is a point $x_T$ on~$\pi_T$ with $d(\pi_x,x_T)\leq \kappa$ for a constant $\kappa$ that depends only on $\delta$, $c_1$, and $c_2$.
Hence we have $d(x,x_T)\leq \kappa+D$ and (iii) is proved.
\end{proof}

\subsection{The case: hyperbolic spaces}

The final aim of the first part of this paper is to demonstrate the tree-likeness of hyperbolic spaces in terms of contained $\real$-trees.
For that we combine the result for the visual hyperbolic spaces with the theorems from Section~\ref{sec_DefVisual}.

Before we can state the main result, we have to make a further definition.
A subset $Y$ of a hyperbolic geodetic space $X$ has {\em finite geodetic out-spread} if every geodesic in $X\sm Y$ has finite length.

\begin{thm}\label{thm_mainPart2}
Let $X$ be a proper hyperbolic geodetic space whose hyperbolic boundary has finite Assouad dimension.
Then there is an $\real$-tree $T\sub X$ such that the canonical map $\gamma$ from $\rand T$ to~$\rand X$ exists and has the following properties.
\begin{enumerate}[{\em (i)}]
\item It is surjective;
\item there is a constant $M<\infty$ depending only on the Assouad dimension of $\rand X$ such that $\gamma\inv(\eta)$ has at most $M$ elements for each $\eta\in\rand X$;
\item there is a constant $\Delta<\infty$ depending only on~$\delta$ and on the Assouad dimension of~$\rand X$ such that the set $X\sm B_D(T)$ has finite geodetic out-spread.
\end{enumerate}
\end{thm}

\begin{proof}
By Corollary~\ref{cor_BuSCor7.1.5}, there is a proper visual hyperbolic geodetic subspace $Y$ of~$X$ which has the property that every geodesic in $X\sm Y$ has finite length, so $X\sm Y$ has finite geodetic out-spread.
In $Y$ there is an $\real$-tree $T$ as in Theorem~\ref{thm_mainPart1}.
For this $\real$-tree the canonical map $\rand T\to\rand X$ exists and is surjective and continuous by property (i) of Theorem~\ref{thm_mainPart1}.
Furthermore, (ii) also holds because it holds for~$T$ and~$Y$.
The remaining part is a consequence of the fact that $X\sm Y$ has already finite geodetic out-spread, so in particular $X\sm \bar{B}_\Delta(T)$ (with the constant $\Delta$ from Theorem~\ref{thm_mainPart1}) has finite geodetic out-spread, because $X\sm Y\supseteq X\sm \bar{B}_\Delta(T)$.
\end{proof}

\section{The topological dimension of the boundary}\label{LowerBoundSection}

Before we prove the main result of this section, we have to define the topological dimension of a topological space $X$.
A {\em refinement} $\UF$ of an open cover $\VF$ of $X$ is an open cover of $X$ such that for every $U\in \UF$ there is a $V\in\VF$ with $U\subseteq V$.
$X$ has {\em topological dimension at most $n$} if every open cover has a refinement such that each $x\in X$ lies in at most $n+1$ elements of the refinement, and $X$ has {\em topological dimension $n$} if it has topological dimension at most $n$ but not topological dimension at most $n-1$.
If there exists no $n\in\nat$ such that $X$ has topological dimension at most $n$ then $X$ has {\em infinite topological dimension}.
Remark, that we always have $\dim X\leq\dim_A X$ by~\cite[Facts~3.3]{Luukainen}.

\begin{lem}\label{totDisToBoundary}
Let $X$ be a compact metric space such that there exists a totally disconnected compact metric space $Y$ and an equivalence relation $\sim$ on~$Y$ with at most $M<\infty$ elements in each equivalence class such that $X$ and $Y/\sim$ are homeomorphic.
Then $X$ has topological dimension at most $M-1$.
\end{lem}

\begin{proof}
Let $\UF$ be a finite critical open cover of~$X$.
Let $\UF'$ be that open cover of $Y$ that is induced by~$\UF$, that is a $U\in\UF$ corresponds to precisely one $U'\in\UF'$ and $y\in U'$ if and only if $[y]\in U$ (where we assume that $X=Y/\sim$).
As $Y$ has topological dimension $0$, there is a finite open cover $\VF'$ of~$\UF'$ with pairwise disjoint elements, since it is a well-know fact, that any totally disconnected compact metric space has topological dimension $0$.
For any $V'\in\VF'$ let $V$ be the set of all $[y]$ with $y\in V'$.
Let $\VF$ be the set of all such sets $V$ for $V'\in\VF'$.
Any $V$ is an open set and thus $\VF$ is an open cover of $X$.
By the construction $\VF$ is also a refinement of $\UF$ and has multiplicity at most $M$.
Thus the topological dimension of~$X$ is at most $M-1$.
\end{proof}

\begin{thm}
Let $X$ be a proper hyperbolic geodetic space and let $T$ be an $\real$-tree. Assume that there exists a canonical map from $\rand T$ to $\rand X$ such that there is an $M<\infty$ such that every hyperbolic boundary point of~$X$ has at most $M$ preimages in $\rand T$.
Then the topological dimension of~$\rand X$ is at most $M-1$.
\end{thm}

\begin{proof}
Since $\rand T$ and $\rand X$ are compact metric spaces and $\rand T$ is totally disconnected, the theorem is a direct consequence of Lemma~\ref{totDisToBoundary}.
\end{proof}

In the terms of~\cite{HyperbolicSpanningTrees} where we investigated spanning trees of locally finite hyperbolic graphs we obtain an analog result.
A {\em spanning tree} of a graph is a subgraph on all vertices of the graph that is a tree.

\begin{thm}
Let $G$ be a locally finite hyperbolic graph and let $T$ be a spanning tree of $G$ such that the canonical map from $\rand T$ onto $\rand G$ exists and such that there is an $M<\infty$ such that any boundary point of $G$ has at most $M$ preimages.
Then the topological dimension of $\rand G$ is at most $M-1$.\qed
\end{thm}

\providecommand{\bysame}{\leavevmode\hbox to3em{\hrulefill}\thinspace}
\providecommand{\MR}{\relax\ifhmode\unskip\space\fi MR }
\providecommand{\MRhref}[2]{%
  \href{http://www.ams.org/mathscinet-getitem?mr=#1}{#2}
}
\providecommand{\href}[2]{#2}

\end{document}